\documentclass[12pt]{amsart}
\usepackage{amssymb,verbatim}
\usepackage{tikz}
\usetikzlibrary{snakes}
\usetikzlibrary{arrows}

\newtheorem{theorem}{Theorem}[section]
\newtheorem{corollary}[theorem]{Corollary}
\newtheorem{lemma}[theorem]{Lemma}
\newtheorem{claim}[theorem]{Claim}
\newtheorem{proposition}[theorem]{Proposition}

\newtheorem{remark}[theorem]{Remark}

\theoremstyle{definition}
\newtheorem{definition}[theorem]{Definition}
\newtheorem{assumption}[theorem]{Assumption}

\newcommand{\To}{\rightarrow}

\newcommand{\orb}{\text{orb}}

\newcommand{\sgn}{\mbox{sgn}}

\newcommand{\bbr}{\mbox{$\mathbb{R}$}}
\newcommand{\bbz}{\mbox{$\mathbb{Z}$}}

\newcommand{\bbn}{\mbox{$\mathbb{N}$}}

\newcommand{\di}{\text{d}}

\newcommand{\dP}{\delta P}

\begin{document}

\title{Code \& order in polygonal billiards}

\author{Jozef Bobok}

\author{Serge Troubetzkoy}

\address{KM FSv \v CVUT, Th\'akurova 7, 166 29 Praha 6, Czech Republic}
\email{bobok@mat.fsv.cvut.cz}

\address{Centre de physique th\'eorique\\
Federation de Recherches des Unites de Mathematique de Marseille\\
Institut de math\'ematiques de Luminy and\\
Universit\'e de la M\'editerran\'ee\\
Luminy, Case 907, F-13288 Marseille Cedex 9, France}
\email{troubetz@iml.univ-mrs.fr}
\urladdr{http://iml.univ-mrs.fr/{\lower.7ex\hbox{\~{}}}troubetz/} \date{}

\begin{abstract}
Two  polygons $P,Q$ are code equivalent if there are billiard orbits $u,v$ which hit the same sequence of sides and such that the projections of the orbits are dense in the boundaries $\partial P, \partial Q$. Our main results show when code equivalent polygons have the same angles, resp. are similar, resp. affinely similar.
\end{abstract}

\maketitle

\section{Introduction}
Consider a simply connected polygon $P$ with $k$ sides.  Code each the billiard orbit by the sequence of sides it hits.  We study the following question: {\em can the arising sequence be realized as the coding of a billiard orbit in another
polygon}? Of course if the orbit is not dense in the boundary $\partial P$, then we can modify $P$ preserving the orbit
by adding sides on the untouched
part of the boundary.  In the case that the orbit is periodic we have even more, there is an open neighborhood of $P$ in a co-dimension one submanifold of the set of all
$k$-gons for which the periodic orbit persists \cite{T}.  It is therefore natural to study this question under the assumption that the orbit
is dense in the boundary. More precisely, we say that two polygons $P,Q$ are code equivalent if there are forward billiard
orbits $u,v$ whose projections to the boundaries $\partial P$, $\partial Q$ are dense.
We study this question under this assumption and under various regularity conditions on the orbit $u$.

We first assume a weak regularity condition, a direction $\theta$ is called non-exceptional if there is no generalized diagonal
in this direction.  All but countably many directions are non-exceptional.  Under this assumption we show that
an irrational polygon can not be code
equivalent to a rational polygon (Theorem \ref{t6}) and if two rational polygons are code equivalent then the
angles at corresponding corners are equal (Theorem \ref{p-sameangles}), for triangles this implies they must be similar
(Corollary \ref{c-triangle}).  Next we assume a stronger regularity condition on the angle, unique ergodicity of the billiard flow
in the direction $\theta$, which is verified for almost every direction in a rational polygon. Under this assumption
we show that two rational polygons which are
code equivalent must be affinely similar and if the greatest common denominator of the angles is at least 3 then they
must be similar (Theorem \ref{t5}, Corollary \ref{cc1}).

In \cite{BT} we proved analogous results under the assumption that $P,Q$ are order equivalent. Our investigation of
code equivalence is motivated by Benoit Rittaud's review article on these results \cite{Ri}.
We compare our results with those of \cite{BT}. We show that under the weak regularity condition order equivalence implies code equivalence (Theorem \ref{t3}), while under
the strong regularity condition they are equivalent (Theorem \ref{converse}, Corollary \ref{equiv}).  
The proof of this equivalence uses  Corollary \ref{cc1}.  We do not know if under the weak regularity
condition code equivalence implies order equivalence.

\section{Polygonal Billiards}\label{spb}

A polygonal billiard table is a polygon $P$.  Our polygons are assumed to be planar, simply connected, not necesarily convex, and compact, with all angles non trivial, i.e. in $(0,2\pi)\setminus \{\pi\}$.
The billiard
flow $\{T_t\}_{t\in\bbr}$ in $P$ is generated by the free motion of
a point mass subject to elastic reflections in the boundary. This
means that the point moves along a straight line in $P$ with a
constant speed until it hits the boundary. At a smooth boundary
point the billiard ball reflects according to the well known law of
geometrical optics: the angle of incidence equals  the angle of
reflection. If the billiard ball hits a corner, (a non-smooth
boundary point), its further motion is not defined. Additionally to
corners, the billiard trajectory is not defined for orbits
tangent to a side.

By $D$ we denote
the group generated by the reflections in the lines through the
origin, parallel to the sides of the polygon $P$. The group $D$ is
either

\begin{itemize}\item finite, when  all the angles of $P$ are of the form $\pi
m_i/n_i$ with distinct co-prime integers $m_i$, $n_i$,  in this
case $D=D_N$ the dihedral group generated by the reflections in
lines through the origin that meet at angles $\pi/N$, where $N$
is the least common multiple of $n_i$'s,\end{itemize} or
\begin{itemize}\item countably infinite, when at least one angle
between sides of $P$ is an irrational multiple of $\pi$.
\end{itemize} In the two cases we will refer to the polygon as
rational, respectively irrational.

Consider the phase space $P\times S^1$ of the billiard flow $T_t$,
and for $\theta\in S^1$, let $R_{\theta}$ be its subset of points
whose second coordinate belongs to the orbit of $\theta$ under $D$.
Since a trajectory changes its direction by an element of $D$ under
each reflection, $R_{\theta}$ is an invariant set of the billiard
flow $T_t$ in $P$. The set $P\times \theta$ will be called a
floor of the phase space of the flow $T_t$.

As usual, $\pi_1$, resp. $\pi_2$
denotes the first natural projection (to the foot
point), resp. the second natural projection (to the
direction).
 A direction, resp.\  a point $u$ from the phase space is
exceptional if it is the direction of a generalized diagonal  (a generalized diagonal is a billiard trajectory that
goes from a corner to a corner), resp.\
$\pi_2(u)$ is such a direction. Obviously there are countably many
generalized diagonals hence also exceptional directions. A
direction, resp.\  a point $u$ from the phase space, which is not
exceptional will be called non-exceptional.

In a rational polygon a billiard trajectory may have only finitely
many different directions. The set $R_{\theta}$ has the structure of a surface.
For non-exceptional $\theta$'s the
faces of $R_{\theta}$ can be glued according to the action of $D_N$
to obtain a flat surface depending only on the polygon $P$ but not
on the choice of $\theta$ - we will denote it $R_P$.

Let us recall the construction of $R_P$. Consider $2N$ disjoint
parallel copies $P_1,\dots,P_{2N}$ of $P$ in the plane. Orient the
even ones clockwise and the odd ones counterclockwise. We will glue
their sides together pairwise, according to the action of the group
$D_N$. Let $0<\theta=\theta_1<\pi/N$ be some angle, and let
$\theta_i$ be its $i$-th image under the action of $D_N$. Consider
$P_i$ and reflect the direction $\theta_i$ in one of its sides. The
reflected direction is $\theta_j$ for some $j$. Glue the chosen side
of $P_i$ to the identical side of $P_j$. After these gluings are
done for all the sides of all the polygons one obtains an oriented
compact surface $R_P$.

Let $p_i$ be the $i$-th vertex of $P$ with the angle $\pi m_i/n_i$
and denote by $G_i$ the subgroup of $D_N$ generated by the
reflections in the sides of $P$, adjacent to $p_i$. Then $G_i$
consists of $2n_i$ elements. According to the construction of $R_P$
the number of copies of $P$ that are glued together at $p_i$ equals
to the cardinality of the orbit of the test angle $\theta$ under the
group $G_i$, that is, equals $2n_i$.

The billiard map $T\colon~V_P=\cup e\times\Theta\subset \dP\times
(-\frac{\pi}{2},\frac{\pi}{2})\To V_P$ associated with the flow $T_t$ is the first return map
to the boundary $\dP$ of $P$. Here the union $\cup e \times \Theta$
is taken over all sides of $P$ and for each side $e$ over the inner
pointing directions $\theta\in\Theta = (-\frac{\pi}{2},\frac{\pi}{2})$ measure with respect to the inner pointing normal.  We will denote points of $V_P$ by $u = (x,\theta)$.

We sometimes use the map $\varrho_1$, $\varrho_2$ and $\varrho$ mapping $(e\times (-\frac{\pi}{2},\frac{\pi}{2}))^2$ into $\bbr^+$ defined as
$\varrho_1(u,\tilde u)=\vert\pi_1(u)-\pi_1(\tilde u)\vert$,  $\varrho_2(u,\tilde u)=\vert\pi_2(u)-\pi_2(\tilde u)\vert$ and  $\varrho=\max\{\varrho_1,\varrho_2\}$. Clearly the map $\varrho$ is a metric.

The bi-infinite (forward, backward) trajectory (with respect to $T$) is not defined for all points from $V_P$.  The set of points from $V_P$ for which the bi-infinite, forward and backward trajectory exists is denoted by $BIV_P$, $FV_P$ and $BV_P$ respectively.

For a simply connected polygon we always consider counterclockwise orientation of its boundary $\dP$.
We denote $[x,x']$ ($(x,x')$) a closed (open) arc with outgoing
endpoint $x$ and incoming endpoint $x'$.

If $P,Q$ are simply connected polygons, two sequences $\{x_n\}_{n\ge 0}\subset \partial  P$ and
$\{y_n\}_{n\ge 0}\subset \partial  Q$ have the same combinatorial order if
for each non-negative integers $k,l,m$
\begin{equation}\label{e14}x_k\in [x_l,x_m]~\iff~y_k\in [y_l,y_m].\end{equation}

We proceed by recalling several well known and useful (for
our purpose) results about polygonal billiards (see for example
\cite{MT}).
Recall that a flat
strip $\mathcal T$ is an invariant subset of the phase space of the
billiard flow/map such that
\begin{itemize}\item[1)] $\mathcal T$ is
contained in a finite number of floors,\item[2)] the billiard
flow/map dynamics on $\mathcal T$ is minimal in the sense that
any orbit which does not hit a corner is dense in $\mathcal T$,
\item[3)] the boundary of $\mathcal
T$ is non-empty and consists of a finite union of generalized
diagonals.\end{itemize}

The set of the corners of $P$ is denoted by $C_P$. As usual, an $\omega$-limit set of a point $u$  is denoted by
$\omega(u)$.

\begin{proposition}\cite{MT}~\label{p-summary}Let $P$ be rational and $u \in FV_P$. Then exactly one of the following three possibilities has to be satisfied.
\begin{itemize}\item[(i)] $u$ is periodic.
\item[(ii)] $\overline{\orb}(u)$ is a flat strip; the billiard flow/map is minimal on $\overline{\orb}(u)$.
\item[(iii)] For the flow $T_t$, $\omega(u)=R_{\pi_2(u)}$. The billiard flow/map is minimal on $R_{\pi_2(u)}$. We have
\begin{equation*}
\#(\{\pi_2(T^n(u))\colon~n\ge 0\})=2N,\end{equation*} and for
every $x\in\partial  P\setminus C_P$,
\begin{equation*}
\#\{u_0\in\omega(u)\colon~\pi_1(u_0)=x\}=N,\end{equation*} where $N=N_P$
is the least common multiple of the denominators of angles of
$P$. Moreover, in this case
\begin{equation*}\pi_2(\{u_0\in\omega(u)\colon~\pi_1(u_0)=x\})=\pi_2(\{u_0\in\omega(u)\colon~\pi_1(u_0)=x'\})\end{equation*}
whenever $x'\notin C_P$ belongs to the same side as $x$. Case (iii) holds whenever $u \in FV_P$ is non-exceptional.
\end{itemize}
\end{proposition}

\begin{corollary}
Let $P$ be rational and $u \in FV_P$, then $u$ is recurrent and the $\omega$-limit set $\omega(u)$ coincides with the
forward orbit closure
$\overline{\orb}(u)$.
\end{corollary}

\begin{theorem}\label{t-irrat}\cite[Theorem 4.1]{BT}~Let $P$ be irrational and $u \in FV_P$.
\begin{itemize}\item[(i)] If $\pi_2(u)$ is non-exceptional then
$\{\pi_2(T^nu)\colon~n\ge 0\}$ is infinite.
\item[(ii)] If $u$ is not periodic, but visits only a finite
number of floors then ($u$ is uniformly recurrent and) $\overline{\orb}(u)$ is
a flat strip.
\end{itemize}
\end{theorem}
Combining Proposition \ref{p-summary} and Theorem \ref{t-irrat}(ii) yields
\begin{corollary}\label{c1}Let $P$ be a polygon and $u\in V_P$ visits a finite number of floors. Then $u$ is uniformly recurrent.\end{corollary}

Let $G$ be a function defined on a neighborhood of $y$. The derived
numbers $D^+G(y)$, $D_+G(y)$ of $G$ at $y$ are given by $$D^+G(y) =
\limsup_{h\to 0_+}\frac{G(y + h)-G(y)}{h},~D_+G(y) = \liminf_{h\to
0_+}\frac{G(y + h)-G(y)}{h}$$ and the analogous limits from the left
are denoted by $D^-G(y)$, $D_-G(y)$.

Let $(z,y)$ be the coordinates of $\bbr^2$ and let $p_{a,b} \subset
\bbr^2$ be the line with equation $y = a + z \tan b$. For short we
denote $p_{y_0,G(y_0)}$ by $p_{G(y_0)}$. The following useful lemma was proven in \cite{BT}.
\begin{lemma}\label{l-infinitesimal}Let $G\colon~(c,d)\to (-\frac{\pi}{2},\frac{\pi}{2})$ be a continuous function. Fix $C\subset (c,d)$ countable.
Assume that for some $y_0$ one of the four possibilities
\begin{equation*}\label{e10}D^+G(y_0)>0,~
D_+G(y_0)<0,~D^-G(y_0)>0,~D_-G(y_0)<0\end{equation*} is fulfilled.
Then there exists a sequence $\{y_n\}_{n\ge 1}\subset (c,d)\setminus
C$ such that $\lim_ny_n=y_0$ and the set of crossing points
$\{p_{G(y_0)}\cap p_{G(y_n)}\colon~n\ge 1\}$ is bounded in the
$\bbr^2$.
\end{lemma}

\section{Coding by sides}
For a simply connected $k$-gon $P$ we always consider counterclockwise
numbering of sides $e_1=[p_1,p_2],\dots,e_k=[p_k,p_1]$; we denote $e_i^{\circ}=(p_i,p_{i+1})$.

The symbolic bi-infinite (forward, backward) itinerary of a point $u = (x,\theta)\in BIV_P$ ($u\in FV_P$, $u\in BV_P$) with respect to the sides of $P$ is a sequence
$\sigma(u)=\{\sigma_i(u)\}_{i=-\infty}^{\infty}$, ($\sigma^+(u)=\{\sigma_i(u)\}_{i\ge 0}$, $\sigma^-(u)=\{\sigma_i(u)\}_{i\le 0}$) of numbers from $\{1,\dots,k\}$ defined by
\begin{equation*}\pi_1(T^iu)\in e^{\circ}_{\sigma_i}.\end{equation*}

Let $\Sigma_P := \{ \sigma^+(u) : \ u \in FV_P\}$.
For a sequence $\sigma=\{\sigma_i\}_{i\ge 0} \in \Sigma_P$ we denote by $X(\sigma)$ the set of points from $V_P$ whose symbolic forward itinerary equals to $\sigma$.

 \begin{theorem}\label{t1}\cite{GKT} Let $P$ be a polygons and $\sigma \in \Sigma_P$ be periodic. Then each point from $X(\sigma)$ has a periodic trajectory.\end{theorem}

We will repeatedly use the following result.

 \begin{theorem}\label{t2}\cite{GKT} Let $P$ be a polygon and $\sigma \in \Sigma_P$ be non-periodic, then the set $X(\sigma)$ consists of one point.\end{theorem}

For $u\in FV_P$ and $m\ge 1$ denote
\begin{equation*}FV_P(u,m)=\{w\in FV_P\colon~\sigma_i(u)=\sigma_i(w),~i=0,\dots,m-1\}\end{equation*}
and define positive numbers $\varepsilon_{i,m}$, $i=1,2$ and $\varepsilon_m$ by
\begin{equation}\label{e13}\varepsilon_{i,m}=\sup\{\varrho_i(w,u)\colon~w\in FV_P(u,m)\},~\varepsilon_m=\max\{\varepsilon_{1,m},\varepsilon_{2,m}\}.\end{equation}

We remind the reader the notion of
an unfolded billiard trajectory. Namely,
instead of reflecting the trajectory in a side of $P$ one may reflect $P$ in
this side and unfold the trajectory to a straight line. As a consequence of Theorem \ref{t2} we obtain

\begin{proposition}\label{p0}~If $u\in FV_P$ is non-periodic then $\lim_m\varepsilon_m=0$.\end{proposition}
\begin{proof} Unfolding billiard trajectories  immediately yields
$\lim_m\varepsilon_{2,m}=0.$
Note that $\varepsilon_m$ is decreasing and assume that $\varepsilon_0=\lim_m\varepsilon_m>0$. Then necessarily also $\varepsilon_0=\lim_m\varrho_1(u,w_m)$ for some $w_m=(x_m,\theta_m)\in FV_P(u,m)$, i.e., $\lim_m x_m=x\in e_{\sigma_0(u)}$ and $\vert\pi_1(u)-x\vert=\varepsilon_0$. Denoting $\tilde x$ the middle of an arc with the endpoints $\pi_1(u),x$, we get $\sigma^+((\tilde x,\pi_2(u))=\sigma^+(u)$, what is impossible by Theorem \ref{t2}.
\end{proof}

We let to the reader the verification of the following fact.

\begin{proposition}\label{p4}~Let $P$ be a polygon. For every $\delta>0$ there exists an $m=m(\delta)\in\bbn$ such that whenever $u,\tilde u\in V_P$ satisfy $\varrho_2(u,\tilde u)>\delta$ and for some $n$, $\vert n\vert\ge m$, the symbols $\sigma_n(u),\sigma_n(\tilde u)$ exist, then $\sigma_n(u)\neq \sigma_n(\tilde u)$.\end{proposition}

An increasing sequence $\{n(i)\}_{i\ge 0}$ of positive integers is called syndetic if the sequence $\{n(i+1)-n(i)\}_{i\ge 0}$ is bounded.
A symbolic itinerary $\sigma^+$ is said to be ({\it uniformly}) recurrent if for every initial word $(\sigma_0,\dots,\sigma_{m-1})$ there is a ({\it syndetic}) sequence $\{n(i)\}_{i\ge 0}$ such that
$(\sigma_{n(i)}, \cdots, \sigma_{n(i) + m - 1}) = (\sigma_0,\cdots,\sigma_{m-1})$ for all $i $.
For a polygon $P$ and billiard map $T\colon~V_P\to V_P$, a point $u=(x,\theta)\in FV_P$ is said to be ({\it uniformly}) recurrent if for every $\varepsilon>0$ there is a ({\it syndetic}) sequence $\{n(i)\}_{i\ge 0}$ such that $$\varrho(T^{n(i)}u,u)<\varepsilon$$ for each $i$.

It is easy to see that a ({\it uniformly}) recurrent point $u$ has a ({\it uniformly}) recurrent symbolic itinerary. It is a consequence of Theorems \ref{t1},\ref{t2} that the opposite implication also holds true.

 \begin{proposition}\label{p1}~Let $P$ be a polygon and $u\in FV_P$. Then $\sigma^+(u)$ is (uniformly) recurrent if and only
 if $u$ is (uniformly) recurrent.\end{proposition}
\begin{proof} Suppose $\sigma^+(u)$ is (uniformly) recurrent.  By Theorem \ref{t1} we are done if $\sigma^+(u)$ is periodic. If it is non-periodic, Proposition \ref{p0} says that $\lim_m\varepsilon_m=0$, where $\varepsilon_m$ were defined in (\ref{e13}).
Choose an $\varepsilon>0$. Then $\varepsilon_m<\varepsilon$ for some $m$ and we can consider a (syndetic) sequence $\{n(i,m)\}_{i\ge 0}$ corresponding to the initial word $(\sigma_0,\dots,\sigma_{m-1})$ of $\sigma^+(u)$. Clearly, $$\varrho(T^{n(i,m)}u,u)\le \varepsilon_m<\varepsilon$$ for each $i$. The converse is clear.
\end{proof}

\section{Code Equivalence}\label{sce}

\begin{definition}\label{d1}
We say that polygons  $P,Q$ are code equivalent if there are points $u\in FV_P$, $v\in FV_Q$ such that
\begin{itemize}

   \item[(C1)] $\overline{\{\pi_1(T^nu)\}}_{n\ge 0}=\partial  P$, $\overline{\{\pi_1(S^nv)\}}_{n\ge
   0}=\partial  Q$,
     \item[(C2)] the symbolic forward itineraries $\sigma^+(u)$, $\sigma^+(v)$ are the same;
  \end{itemize}
  the points $u,v$ will be sometimes called the leaders.
 \end{definition}

Clearly any two rectangles are code equivalent, and also two code equivalent polygons $P,Q$ have the same number of sides. In this case we always consider their counterclockwise numbering $e_1=[p_1,p_2],\dots,e_k=[p_k,p_1]$ for $P$, resp. $f_1=[q_1,q_2],\dots,f_k=[q_k,q_1]$ for $Q$. We sometimes write $e_i\sim f_i$ to emphasize the correspondence of sides $e_i,f_i$.
The verification that this relation is reflexive, symmetric and
transitive is left to the reader.

 \begin{definition}\label{d3}Let $P$ be a polygon and $u,\tilde u\in FV_P$. We say that trajectories of $u,\tilde u$ intersect before their symbolic separation if either
 \begin{itemize}\item[(p)] for some positive integer $\ell$,  $\sigma_{\ell}(u)\neq\sigma_{\ell}(\tilde u)$, $$\sigma_k(u)=\sigma_k(\tilde u)~\text{ whenever }k\in\{0,\dots,\ell-1\}$$  and for some $k_0\in\{0,\dots,\ell-1\}$, the segments with endpoints $$\pi_1(T^{k_0}u),\pi_1(T^{k_0+1}u)\text{ and }\pi_1(T^{k_0}\tilde u),\pi_1(T^{k_0+1}\tilde u)$$ intersect; or \item[(n)] for some negative integer $\ell$,  $\sigma_{\ell}(u)\neq\sigma_{\ell}(\tilde u)$, $$\sigma_k(u)=\sigma_k(\tilde u)~\text{ whenever }k\in\{\ell+1,\dots,0\}$$  and for some $k_0\in\{\ell,\dots,-1\}$, the segments with endpoints $$\pi_1(T^{k_0}u),\pi_1(T^{k_0+1}u)\text{ and }\pi_1(T^{k_0}\tilde u),\pi_1(T^{k_0+1}\tilde u)$$ intersect.\end{itemize}
\end{definition}
 For $u\in FV_P$, a side $e$ of $P$ and $\theta\in (-\frac{\pi}{2},\frac{\pi}{2})$ we put

\begin{equation}\label{e16}I(u,e,\theta)=\{n\in\bbn\cup\{0\}\colon~\pi_1(T^nu)\in e,~\pi_2(T^nu)=\theta\}.\end{equation}

Throughout the section let $u_n=T^nu$, $x_n=\pi_1(u_n)$, $v_n=S^nv$, $y_n=\pi_1(v_n)$.
\begin{proposition}\label{p5}
Let polygons $P,Q$ be code equivalent with leaders $u,v$, $u$ recurrent. For any $m,n\in I(u,e,\theta)$, the trajectories of $S^mv,S^nv$ cannot intersect before their symbolic separation.\end{proposition}

\begin{figure}[h]
\begin{minipage}[ht]{0.49\linewidth}
\centering
\begin{tikzpicture}[scale=0.7]
  \draw [very thick] (0,0)
    -- node [below] {} ++ (-2,0)
    -- node [left] {} ++ (0,2)
    -- node [below] {} ++ (2,0)
    -- node [below] {} ++ ( 0,-2 )
    -- node [below] {} ++ (2,0)
    -- node [below] {} ++ (0,2)
    -- node [below] {} ++ (-2,0)
    -- node [left] {} ++ (0,2)
    -- node [below] {} ++ (2,0)
    -- node [below] {} ++ (0,-2);
    \draw [very thick] (0,4)
-- node [below]{} (0,6)
-- node [below]{} (2,6)
-- node [below]{} (2,4);  

  \draw [very thick, ->>] (2,5.52) -- (-0.3,0);
  \draw [very thick, ->>]  (1.8,6) -- (-0.7,0);

 \draw (-0.7,-.5) node {\Tiny\hskip-0.5cm$x_{m(1)}$};
 \draw (0.4,-.5) node {\Tiny\hskip-0.5cm$x_{n(1)}$};
 \draw (1.8,6.5) node {\Tiny\hskip-0.5cm$x_{m(1)+\ell}$};
  \draw (3.3,5.35) node {\Tiny\hskip-0.5cm$x_{n(1)+\ell}$};

       \end{tikzpicture}
       \end{minipage}
\begin{minipage}[ht]{0.49\linewidth}
\centering
\begin{tikzpicture}[scale=0.7]
  \draw [very thick] (0,0)
    -- node [below] {} ++ (-2,0)
    -- node [left] {} ++ (0,2)
    -- node [below] {} ++ (2,0)
    -- node [below] {} ++ ( 0,-2 )
    -- node [below] {} ++ (2,0)
    -- node [below] {} ++ (0,2)
    -- node [below] {} ++ (-2,0)
    -- node [left] {} ++ (0,2)
    -- node [below] {} ++ (2,0)
    -- node [below] {} ++ (0,-2);
  \draw [very thick] (0,4)
-- node [below]{} (0,6)
-- node [below]{} (2,6)
-- node [below]{} (2,4);  

  \draw [very thick, ->>] (2,5.52) --  (-0.7,0); 
  \draw [very thick, ->>]  (1.8,6) --(-0.3,0);

     \draw (-0.7,-.5) node {\Tiny\hskip-0.5cm$y_{m(1)}$};
 \draw (0.4,-.5) node {\Tiny\hskip-0.5cm$y_{n(1)}$};
 \draw (1.8,6.5) node {\Tiny\hskip-0.5cm$y_{n(1)+\ell}$};
  \draw (3.3,5.35) node {\Tiny\hskip-0.5cm$y_{m(1)+\ell}$};

       \end{tikzpicture}
       \end{minipage}
\caption{\mbox{Parallel versus Crossing with $k_0 = -2$}, the $(n)$-increasing case}
\label{fig2}\end{figure}

\begin{proof} The case when $x_m<x_n$ and $y_m<y_n$, resp. $y_n<y_m$ will be called increasing, resp. decreasing. Thus, using the two parts of Definition \ref{d3} and assuming that the conclusion is not true we can distinguish the following four possibilities: (p)-increasing, (p)-decreasing, (n)-decreasing and (n)-increasing. Let us prove the (n)-increasing case. In this case there are $m,n\in I(u,e,\theta)$, some negative $\ell,k_0$ such that
$$x_m<x_n,~y_m<y_n$$
 and the second part (n) of Definition \ref{d3} is fulfilled. 
 
 Note that we have only assumed that the forward iterates of $x$ and $y$ have the same code, but in the $(n)$-increasing
 case we want to exclude the intersection of their backwards orbits.  We overcome this problem by approximating $x_m$ and $y_n$
 by their forward orbits.  This can be done since the leader $u$ is recurrent, hence  by Proposition \ref{p1}  $v$ is also recurrent.
We consider sufficiently large integers $m(1),n(1)\in (-\ell,\infty)$ such that $v_{m(1)}$, resp. $v_{n(1)}$ approximates $v_m$, resp. $v_n$. Then $\sigma_{m(1)+\ell}(v)\neq\sigma_{n(1)+\ell}(v)$, $\sigma_{m(1)+k}(v)=\sigma_{n(1)+k}(v)$ whenever $k\in\{\ell+1,\dots,0\}$; since for some $k_0\in\{\ell,\dots,-1\}$, the segments with endpoints $$\pi_1(T^{m(1)+k_0}v),\pi_1(T^{m(1)+k_0+1}v)\text{ and }\pi_1(T^{n(1)+k_0}v),\pi_1(T^{n(1)+k_0+1}v)$$ intersect and the points $u_{m(1)},u_{n(1)}$ are (almost) parallel, we get
$$\sgn~(\sigma_{m(1)+\ell}(u)-\sigma_{n(1)+\ell}(u))\neq \sgn~(\sigma_{m(1)+\ell}(v)- \sigma_{n(1)+\ell}(v)),$$
what is not possible for the leaders $u,v$. The other three cases are analogous.\end{proof}

In the last part of this section we present Corollaries \ref{c2}-\ref{cl} of Proposition \ref{p5} under the following

\begin{assumption}\label{a1}Let $P$,$Q$ be code equivalent polygons with leaders $u,v$ and the set of directions $\{\pi_2(T^nu)\colon~n\ge 0\}$ along the trajectory of $u$ is finite.\end{assumption}

When proving Corollaries \ref{c2}-\ref{cl} we denote    $\alpha_n=\pi_2(u_n)$, $\beta_n=\pi_2(v_n)$. By Definition \ref{d1}(C1) the first projection of the
forward trajectory of $u$, resp. of $v$ is dense in $\partial  P$, resp. $\partial Q$, so in particular,
neither $u$ nor $v$ is periodic. In any case, the set $I(u,e,\theta)$ defined for a side $e=e_i$ in (\ref{e16})
is nonempty only for $\theta$'s from the set $\{\pi_2(T^nu): n \ge 0 \}$ which is assumed to be finite.
In what follows we fix such $e$ and $\theta$.

Applying Corollary \ref{c1} and Proposition \ref{p1} we obtain that both the leaders $u$ and $v$ are uniformly recurrent.

Obviously the set
\begin{equation}\label{e12}\mathcal J(e,\theta)=\overline{\{y_n\colon~n\in I(u,e,\theta)\}}\end{equation}
is a perfect subset of a side $f=f_i\sim e$. The counterclockwise orientation of $\partial Q$ induces the linear ordering of $f$ and we can consider two elements $\min \mathcal J(e,\theta),\max \mathcal J(e,\theta)\in f$.

Define a function $g\colon~\{y_n\}_{n\in I(u,e,\theta)}\to (-\frac{\pi}{2},\frac{\pi}{2})$ by $g(y_n)=\beta_n$.

\begin{corollary}\label{c2}The function $g$ can be extended continuously to the map $G\colon~\mathcal J(e,\theta)\to [-\frac{\pi}{2},\frac{\pi}{2}]$. Moreover, $G(y)\in (-\frac{\pi}{2},\frac{\pi}{2})$ for each $$y\in \mathcal J(e,\theta) \setminus \{\min \mathcal J(e,\theta),\max \mathcal J(e,\theta)\}.$$\end{corollary}
\begin{proof} Put $G(y_n)=\beta_n$.  Proposition \ref{p5} clearly shows that for $n(k)\in I(u,e,\theta)$,\begin{equation*}\label{e20}y_{n(k)}\rightarrow_k y\in \mathcal J(e,\theta)~\text{ implies }\beta_{n(k)}\to \beta\in [-\frac{\pi}{2},\frac{\pi}{2}]\end{equation*} and we can put $G(y)=\beta$.

 Let $y\in \mathcal J(e,\theta) \setminus \{\min \mathcal J(e,\theta),\max \mathcal J(e,\theta)\}$ and choose $y_{n(i)},y_{n(j)}$ such that \begin{equation}\label{e15}y\in (y_{n(i)},y_{n(j)}).\end{equation}

If $G(y)=-\frac{\pi}{2}$, resp. $G(y)=\frac{\pi}{2}$ then by (\ref{e15}) and the continuity of $G$, for some $v_{n(k)}$ sufficiently close to $(y,-\frac{\pi}{2})$, resp. $(y,\frac{\pi}{2})$, the trajectories of $v_{n(j)},v_{n(k)}$, resp.  $v_{n(i)},v_{n(k)}$ intersect before their symbolic separation, what contradicts Proposition \ref{p5}. Thus $G(y)\in (-\frac{\pi}{2},\frac{\pi}{2})$.
\end{proof}

The notion of combinatorial order has been introduced in (\ref{e14}).

\begin{corollary}\label{c4}The sequences $\{x_n\}_{n\in I(u,e,\theta)}\subset e$ and $\{y_n\}_{n\in I(u,e,\theta)}\subset f$ have the same combinatorial order.\end{corollary}
\begin{proof}The conclusion is true when $\# I(u,e,\theta)\le 1$. Assume to the contrary that for some $m,n\in I(u,e,\theta)$,\begin{equation*}\label{e21}x_m<x_n~\text{ and }~y_n<y_m.\end{equation*}Since by Proposition \ref{p5} the trajectories of $v_m,v_n$ cannot intersect before their symbolic separation, $$\sgn~(\sigma_k(u_m)-\sigma_k(u_n))\neq \sgn~(\sigma_k(v_m)-\sigma_k(v_n))$$ for some $k\in\bbn$, what is not possible for the leaders $u,v$. The case $x_n<x_m$ and $y_m<y_n$ can be disproved analogously.\end{proof}

Since \begin{equation*}
\bigcup_{e,\theta}\mathcal J(e,\theta)=\partial Q,
\end{equation*}
where the number of summands on the left is by Assumption \ref{a1} finite, Baire's theorem \cite[Theorem 5.6]{Ru} implies that there exists a side $e$ and an angle $\theta$ for which $\mathcal J(e,\theta)$ has a nonempty interior. Denote $[c,d]$ a nontrivial connected component of $\mathcal J(e,\theta)$. Put \begin{equation*}\tau=\{(y,G(y))\colon~y\in [c,d]\}.\end{equation*}
\begin{corollary}\label{c5}There is a countable subset $\tau_0$ of $\tau$ such that each point
from $\tau\setminus \tau_0$ has a bi-infinite
trajectory (either the forward or backward trajectory starting from
any point of $\tau_0$ finishes in a corner of $Q$).\end{corollary}
\begin{proof}Assume that there are two points $\hat v,\tilde v\in\tau$ such that $\pi_1(\hat v)<\pi_1(\tilde v)$, for some $k\in\bbn$ $\pi_1(S^k\hat v)=\pi_1(S^k\tilde v)$ is a common corner and $\sigma_i(\hat v)=\sigma_i(\tilde v)$ for $i\in\{0,\dots,k-1\}$.
As before let $v_n = S^nv$. Choose three of these points $v_{\ell},v_m,v_n\in\tau$ satisfying
\begin{itemize}\item $\pi_1(v_m)<\pi_1(v_{\ell})<\pi_1(v_n)$\item $v_m$, resp. $v_n$ is (sufficiently) close to $\hat v$, resp. $\tilde v$\end{itemize}
Then the trajectories of either $v_m,v_{\ell}$ or $v_{\ell},v_n$ intersect before their symbolic separation,  contradicting
Proposition \ref{p5}.

Thus for each $k \ge 1$ and each $\hat v, \tilde v \in \tau$ with common symbolic itenerary of length $k$, we can not have $\pi_1(S^k \hat v) = \pi_1(S^k \tilde v)$ is a corner,
or equivalently each corner can have at most one preimage of order $k$ for each forward symbolic
itinerary segment of length $k$. This implies that the set $\tau_{0,F}=\tau\setminus FV_Q$ is at most countable. This is also true for $\tau_{0,B}=\tau\setminus BV_Q$ and we can put $\tau_0=\tau_{0,F}\cup \tau_{0,B}$.\end{proof}

\begin{corollary}\label{c6}The continuous function $G\colon~\mathcal J(e,\theta)\to [-\frac{\pi}{2},\frac{\pi}{2}]$ defined in Corollary \ref{c2} has to be constant on each connected component $[c,d]$ of $\mathcal J(e,\theta)$.\end{corollary}\begin{proof} Since by Corollary \ref{c5} the projection $C=\pi_1(\tau_0)$ is countable and $G$ is continuous, it is
sufficient to show that $G'(\tilde y_0)=0$ whenever $\tilde y_0\in (c,d)\setminus
C$.

 To simplify the notation, choose the origin of $S^1$ to be
the direction perpendicular to the side of $Q$ containing $(c,d)$
and  fix $\tilde y_0\in (c,d)\setminus C$; then by Corollary \ref{c2} for a sufficiently small
neighborhood $U(\tilde y_0)$ of $\tilde y_0$, $G(U(\tilde y_0))\subset
(-\frac{\pi}{2},\frac{\pi}{2})$.

For $\tilde y \in U(\tilde y_0) \setminus C$ consider the unfolded (bi-infinite)
billiard trajectory of $(\tilde y,G(\tilde y))$ under the billiard flow
$\{S_t\}_{t\in\bbr}$ in $Q$. Via unfolding, this trajectory corresponds to the line
$p_{G(\tilde y)}$ with the equation $y=\tilde y+z\tan G(\tilde y)$.

\begin{claim}\label{cl}There is no sequence $\{\tilde y_n\}_{n\ge 1}\subset
(c,d)\setminus C$ such that $\lim_n\tilde y_n=\tilde y_0$ and the set of crossing
points $\{p_{G(\tilde y_0)}\cap p_{G(\tilde y_n)}\colon~n\ge 1\}$ is
bounded.\end{claim}
\begin{proof}Assuming the contrary of the conclusion we can consider sufficiently large $n$ and some point $v_k$, resp. $v_{\ell}$ approximating $(\tilde y_0,G(\tilde y_0))$, resp. $(\tilde y_n,G(\tilde y_n))$ such that the trajectories of $v_k,v_{\ell}$ intersect before their symbolic separation, what is impossible by Proposition \ref{p5}.
\end{proof}
Now, applying Lemma \ref{l-infinitesimal} and Claim \ref{cl} we obtain
that the function $G$ satisfies
$G'(\tilde y_0)=0$ for every $\tilde y_0\in (c,d)\setminus C$, i.e., for some $\vartheta\in (-\frac{\pi}{2},\frac{\pi}{2})$, $G\equiv\vartheta$ is constant on $[c,d]$.

\end{proof}

\section{Rational versus Irrational}
\begin{lemma}\label{l-finitefloors}Let $P$,$Q$ be code equivalent with leaders $u,v$;  $P$ rational. Then the set of directions $$\{\pi_2(S^nv)\colon~n\ge 0\}$$
along the trajectory of $v$ is finite.
\end{lemma}
\begin{proof}Applying Lemma \ref{l-infinitesimal} and Corollary \ref{cl} we obtain
that the function $G$ defined in Corollary \ref{c2} satisfies
$G'(\tilde y_0)=0$ for every $\tilde y_0\in (c,d)\setminus C$, i.e., for some $\vartheta\in (-\frac{\pi}{2},\frac{\pi}{2})$, $G\equiv\vartheta$ is constant on $[c,d]$, where $[c,d]$ is a nontrivial connected component of $\mathcal J(e,\theta)$ defined in (\ref{e12}).

We know that the leader $v$ is uniformly recurrent. Take a positive integer $n\in I(u,e,\theta)$ and a positive $\varepsilon_0$ such that $$(y_n-\varepsilon_0,y_n+\varepsilon_0)\subset (c,d).$$ There is a syndetic sequence $\{n(i)\}_{i\ge 0}\subset I(u,e,\theta)$ for which
$$\varrho(S^{n(i)}v_n,v_n)<\varepsilon_0,~\pi_2(S^{n(i)}v_n)=\vartheta$$ for each $i$. This shows that the set of directions $\{\pi_2(S^nv)\colon~n\ge 0\}$ along the trajectory of $v$ is
finite.\end{proof}

\begin{remark}In Lemma \ref{l-finitefloors} we do not assume that $u$ is non-exceptional. \end{remark}

\begin{theorem}\label{t6}Let $P$, $Q$ be code equivalent with leaders $u,v$; $P$ rational, $u$ non-exceptional.
Then $Q$ is rational with $v$ non-exceptional.
\end{theorem}
\begin{proof}
As before, we put $x_n=\pi_1(T^nu)$ and $y_n=\pi_1(S^nv)$. Assume $v$ is exceptional. By Definition \ref{d1} $v$ is non-periodic. At the same time Lemma \ref{l-finitefloors} says that the set of directions \begin{equation*}\label{e19}\{\pi_2(S^nv)\colon~n\ge 0\}\end{equation*}
along the trajectory of $v$ is finite. Thus Proposition \ref{p-summary} implies that
 if $Q$ is rational then $v$ is minimal in a flat strip or in an invariant surface
$R_{\pi_2(v)}$.  On the other hand if $Q$ is irrational, Lemma \ref{l-finitefloors} and Theorem \ref{t-irrat}(ii) imply that $v$ is
minimal in a flat strip.

Suppose that $v$ is exceptional, then it is parallel to a generalized diagonal $d$ which is the boundary of a minimal flat strip.
The minimality implies that
$v$ is not only parallel to $d$, but $v$ also approximates $d$. Denote $y$, resp.\  $y'$ an outgoing,
resp.~incoming corner of $d$ with \begin{equation}\label{e17}y'=\pi_1(S^{\ell}(y,\beta))\end{equation} for
some $\ell\in\bbn$ and a direction $\beta$ with respect to a side $f=f_i=[q_i,q_{i+1}]$. Let us assume that $y=q_i$
and that $v$ approximates $d$ from the side $f$ (the case when  $v$ approximates $d$ from the other side, i.e.~$y=q_{i+1}$ is similar).
Since $v$ approximates $d$ and the set $\{\pi_2(S^nv)\colon~n\ge 0\}$ is finite we can consider a sequence
$\{n(k)\}_{k\ge 0}$ such that for each $k$,

$$S^{n(k)}v=(y_{n(k)},\beta),~y_{n(k)}\in f,$$

$$S^{n(k)+\ell}v=(y_{n(k)+\ell},\beta'),~y_{n(k)+\ell}\in f',$$ $\lim_{k\to\infty}y_{n(k)}=y$ and $\lim_{k\to\infty}y_{n(k)+\ell}=y'$,
where $\ell$ is given by (\ref{e17}) and $f'$ is the appropriate side of $Q$ with endpoint $y'$.

Let $e=e_i=[p_i,p_{i+1}]$, resp.~$e'$ be the sides of $P$ corresponding to $f$, resp.~$f'$ . Since $P$ is rational, we can assume that $\{n(k)\}_{k\ge 0}\subset I(u,e,\alpha)$ and $\{n(k)+\ell\}_{k\ge 0}\subset I(u,e',\alpha')$ for some $\alpha,\alpha'\in (-\frac{\pi}{2},\frac{\pi}{2})$ and Corollary \ref{c4} can be used. By that corollary the combinatorial order of the sequences  $\{x_n\}_{n\in I(u,e,\alpha)}\subset e$ and $\{y_n\}_{n\in I(u,e,\alpha)}\subset f$, resp. $\{x_n\}_{n\in I(u,e',\alpha')}\subset e'$ and $\{y_n\}_{n\in I(u,e',\alpha')}\subset f'$ are the same. We assume the leader $u$ to be non-exceptional hence by Proposition \ref{p-summary}, the sequence  $\{x_n\}_{n\in I(u,e,\alpha)}$, resp. $\{x_n\}_{n\in I(u,e',\alpha')}$ is dense in the side $e$, resp. $e'$. Then necessarily
$\lim_{k\to\infty}x_{n(k)}=x\in C_P\cap e$ and $\lim_{k\to\infty}x_{n(k)+\ell}=x'\in C_P\cap e'$, hence
 \begin{equation*}x'=\pi_1(T^{\ell}(x,\alpha)),\end{equation*}
what contradicts our choice of non-exceptional $u$. Thus, the leader $v$ has to be non-exceptional.

In order to verify that $Q$ is rational, one can simply use Theorem \ref{t-irrat}(i) and Lemma \ref{l-finitefloors}.
\end{proof}

\section{Rational versus Rational - Preparatory Results}
Throughout this section we will assume that $P,Q$ are rational and code equivalent with non-exceptional leaders $u,v$,
Theorem \ref{t6} implies that the assumption that $v$ is non-exceptional is redundant.

\begin{lemma}\label{l2}  Let $P,Q$ rational be code equivalent with non-exceptional leaders $u,v$.  For every side $e_i$ and every direction $\theta\in\pi_2((e_i\times (-\frac{\pi}{2},\frac{\pi}{2}))\cap\omega(u))$ there exists a direction $\vartheta\in \pi_2((f_i\times (-\frac{\pi}{2},\frac{\pi}{2}))\cap\omega(v))$ such that $I(u,e_i,\theta)=I(v,f_i,\vartheta)$ and the sequences $\{\pi_1(T^nu)\}_{n\in I}$ and $\{\pi_1(S^nv)\}_{n\in I}$ have the same combinatorial order.
\end{lemma}
\begin{proof}Let us fix  a side $e_i$ and a direction $\theta\in\pi_2((e_i\times (-\frac{\pi}{2},\frac{\pi}{2}))\cap\omega(u))$. Using Corollaries \ref{c2}, \ref{c6} we obtain for some $\vartheta\in\pi_2((f_i\times (-\frac{\pi}{2},\frac{\pi}{2}))\cap\omega(v))$
$$I(u,e_i,\theta)\subset I(v,f_i,\vartheta);$$
starting from $f_i$, $\vartheta$ we get $I(u,f_i,\vartheta)\subset I(v,e_i,\theta)$ hence $I=I(u,f_i,\vartheta)=I(v,e_i,\theta)$. The fact that the sequences  $\{\pi_1(T^nu)\}_{n\in I}$ and $\{\pi_1(S^nv)\}_{n\in I}$ have the same combinatorial order is a direct consequence of Corollary \ref{c4}.
\end{proof}

Proposition \ref{p-summary} and Lemma \ref{l2} easily yield
\begin{corollary}\label{l4}  Let $P,Q$ rational be code equivalent with non-exceptional leaders $u,v$.  Then $N_P=N_Q$.
\end{corollary}

\begin{lemma}\label{l1} Suppose $P,Q$ are rational and code equivalent with non-exceptional leaders $u$, $v$; let
$\sigma^+=\{\sigma_k\}_{k\ge 0}$ denote their
 common itinerary.  If $\sigma_m=\sigma_n$ then \begin{equation*}\label{e3}\pi_2(T^mu)<\pi_2(T^nu)\iff \pi_2(S^mv)<\pi_2(S^nv).\end{equation*} \end{lemma}
\begin{proof} As before we denote $u_n=T^nu$, $x_n=\pi_1(u_n)$,  $v_n=S^nv$, $y_n=\pi_1(v_n)$.

Let $\sigma_m=\sigma_n=i\in\{1,\dots,k\}$ for some $m,n\in\bbn\cup\{0\}$; it follows from Lemma \ref{l2} that $\theta^1=\pi_2(u_m)\neq \pi_2(u_n)=\theta^2$ if and only if $\vartheta^1=\pi_2(v_m)\neq \pi_2(v_n)=\vartheta^2$. If our conclusion does not hold we necessarily have
\begin{equation}\label{e9}-\frac{\pi}{2} <\theta^1<\theta^2< \frac{\pi}{2} \text{ and } -\frac{\pi}{2} < \vartheta^2<\vartheta^1<\frac{\pi}{2}.\end{equation}

By Proposition \ref{p-summary}, each of the two sequences \begin{equation*}X_j=\{x_n\colon~n\in I(u,e_i,\theta^j)\},~j\in\{1,2\}\end{equation*} is dense in $e_i$ and an analogous statement is true for $$Y_j=\{y_n\colon~n\in I(u,f_i,\vartheta^j)\},~j\in\{1,2\}.$$ Moreover, from Lemma \ref{l2} we know that the sequences $X_j$ and $Y_j$, $j\in\{1,2\}$ have the same combinatorial order.

Let $m=\max\{m(\vert\theta_1-\theta_2\vert),m(\vert\vartheta_1-\vartheta_2\vert)\}$ due to Proposition \ref{p4}. To a given $\varepsilon>0$ one can consider integers $m(1),m(2)\in (m,\infty)$, $m(1)<m(2)$, for which \begin{equation*}\label{e5}x_{m(1)},x_{m(2)}\in [p_i,p_i+\varepsilon],~ \pi_2(u_{m(1)})=\theta^1,~\pi_2(u_{m(2)})=\theta^2\end{equation*} and also \begin{equation*}\label{e11}y_{m(1)},y_{m(2)}\in [q_i,q_i+\varepsilon],~\pi_2(v_{m(1)})=\vartheta^1,~ \pi_2(v_{m(2)})=\vartheta^2.\end{equation*}

Then $\sigma_{m(1)-m}(u)$, $\sigma_{m(2)-m}(u)$, resp. $\sigma_{m(1)-m}(v)$, $\sigma_{m(2)-m}(v)$ exist and by Proposition \ref{p4} they are different. From (\ref{e9}) we get
\begin{equation*}\sgn~(\sigma_{m(1)-m}(u)-\sigma_{m(2)-m}(u))\neq \sgn~ (\sigma_{m(1)-m}(v)-\sigma_{m(2)-m}(v)),\end{equation*}
what is impossible for the leaders $u,v$, a contradiction.
\end{proof}

For a polygon $P$ and its corner $p_j\in C_P$, an element $w\in V_P$ points at $p_j$ if $\pi_1(Tw)=p_j$. For $u\in FV_P$ we denote $N(u,p_j)$ the number of elements from $\omega(u)$ that point at $p_j$.

\begin{lemma}\label{l5}  Let $P,Q$ rational be code equivalent with non-exceptional leaders $u,v$.  Then $N(u,p_j)=N(v,q_j)$, $1\le j\le k$, where $k$ is a common number of sides of $P,Q$.\end{lemma}
\begin{proof} Let $(x,\theta)\in \omega(u)$ point at $p_j$, $x\in e=e_i$. Since $u$ is non-exceptional, $(x,\theta)\in BV_P$ is not periodic and it is a bothside limit of
$\{T^nu\}_{n\in I}$, where $I=I(u,e,\theta)$. Using Lemma \ref{l2} we can consider a direction $\vartheta\in \pi_2((f_i\times (-\frac{\pi}{2},\frac{\pi}{2}))\cap\omega(v))$ such that $I(u,e,\theta)=I(v,f,\vartheta)$ and the (dense) sequences $\{\pi_1(T^nu)\}_{n\in I}$, $\{\pi_1(S^nv)\}_{n\in I}$ have the same combinatorial order. Clearly, there is a unique element $(y,\vartheta)\in\omega(v)$ (with the same address as $(x,\theta)$) pointing at $q_j$ and satisfying $(y,\vartheta)\in BV_Q$, $\sigma^-((x,\theta))=\sigma^-((y,\vartheta))$. The last equality and Theorem \ref{t2} imply $N(u,p_j)\le N(v,q_j)$. The
argument is symmetric, thus we obtain $N(u,p_j)=N(v,q_j)$. \end{proof}

\section{Rational versus Rational - Main Results}\label{srr}

Let $A(p) \in (0,2\pi)\setminus \{\pi\}$ denote the angle at the corner $p \in C_P$.
\begin{theorem}\label{p-sameangles}Let $P,Q$ be code equivalent with leaders $u,v$; $P$ rational, $u$
non-exceptional. Then $A(p_i)=A(q_i)$, $1\le i\le k$.
\end{theorem}

\begin{proof} Theorem \ref{t6} implies that also $Q$ is
rational with  a non-exceptional leader $v$. Let $k=\# C_P=\#C_Q$; Since $P$, $Q$ are rational and
simply connected, $A(p_i)=\pi m_i^P/n_i^P$ and $A(q_i)=\pi
m_i^Q/n_i^Q$, where $m_i^P$, $n_i^P$, resp.\  $m_i^Q$, $n_i^Q$ are
coprime integers. In what follows, we will show that $n_i^P=n_i^Q$
and $m_i^P=m_i^Q$.

We know from Corollary \ref{l4} that $N_P=N_Q=N$. Thus,
both rational billiards correspond to the same dihedral group $D_N$.

Second, consider the local picture around the $i$th vertex $p_i$.
Denote the two sides which meet at $p_i$ by $e$ and $e'$. Suppose
there are $2n^P_i$ copies of $P$ which are glued at $p_i$. Enumerate
them in a cyclic counterclockwise fashion $1,2,\dots,2n^P_i$. Since
$u$ is non-exceptional its orbit is minimal, so it visits each of
the copies of $P$ glued at $p_i$. In particular the orbit crosses
each of the gluings (copy $j$ glued to copy $j+1$).

Now consider the orbit of $v$. We need to show that there are the
same number of copies of $Q$ glued at $q_i$. Fix a
$j\in\{1,\dots,2n^P_i\}$ viewed as a cyclic group. Since $u$ is
non-exceptional the orbit of $u$ must pass from copy $j$ to copy
$j+1$ of $P$ or vice versa from copy $j+1$ to copy $j$. Suppose that
we are at the instant that the orbit $u$ passes from copy $j$ to
copy $j+1$ of $P$. At this same instant the orbit of $v$ passes
through a side. We label the two copies of $Q$ by $j$ and $j+1$
respectively. This labeling is consistent for each crossing from $j$
to $j+1$.

Since this is true for each $j$, the combinatorial data of the orbit
$u$ glue the corresponding $2n^P_i$ copies of $Q$ together in the
same cyclic manner as the corresponding copies of $P$. Note that the
common point of the copies of $Q$ is a common point of $f$ and
$f'$ - the sides of $Q$ corresponding to $e,e'$ - thus it is necessarily the point $q_i$. In
particular, since Lemma \ref{l1} applies, we have
$2n^P_i$ copies of $Q$ glued around $q_i$ to obtain an angle which
is a multiple of $2\pi$. Thus $2n^Q_i$ must divide $2n^P_i$. The
argument is symmetric, thus we obtain $2n^P_i$ divides $2n^Q_i$. We
conclude that $n^P_i=n^Q_i$.

 Third, let us show that $m_i^Q= m^P_i$. Realizing the gluing of
 $2n_i^P$ copies of $P$ together at $p_i$ we get a point $p\in R_P$ with total angle of
 $2\pi m_i^P$. If $m_i^P>1$, the point $p$ is a cone angle
$2\pi m_i^P$ singularity. In any case, for the direction $\theta$
and the corresponding constant flow on $R_P$, there are
$m_i^P$ incoming trajectories
that enter $p$ on the surface $R_P$, hence also $m_i^P$
points in $V_P$ that finish their trajectory after the first iterate
at the corner $p_i$. Repeating all arguments for $Q$ and
$\vartheta=\pi_2(v)$, one obtain $m_i^Q$ points in $V_Q$ that
finish their trajectory after the first iterate at the corner
$q_i$. Since such a number has to be preserved by Lemma \ref{l5}, the inequality $m_i^P\neq m_i^Q$ contradicts
our assumption that $P$ and $Q$ are code equivalent. Thus, $m_i^Q= m^P_i$.
\end{proof}

A triangle is determined (up to similarity) by its angles, thus
Theorem \ref{p-sameangles} implies
\begin{corollary}\label{c-triangle}
Let $P,Q$ be code equivalent with leaders $u,v$, $P$ a rational triangle, $u$
non-exceptional. Then $Q$ is similar to $P$.
\end{corollary}

For a $P$ rational, the union of edges of $R_P$ - we call it the skeleton of $R_P$ - will be denoted by $K_P$.

It follows from Proposition \ref{p-summary} that for $P$ rational with $u\in V_P$ non-exceptional, $\omega(u)=K_P$.

\begin{proposition}\label{p3}Let $P$, $Q$ be code equivalent with leaders $u,v$; $P$ rational, $u$ non-exceptional. The map $\Psi\colon~\orb(u)\to \orb(v)$ defined by $\Psi(T^nu)=S^nv$, $n\in\bbn\cup\{0\}$ can be extended to the homeomorphism $\Phi\colon~K_P\to K_Q$ satisfying (for all
$n\in\bbz$ for which the image is defined)
\begin{equation*}\Phi(T^n\tilde u)=S^n\Phi(\tilde u),~\tilde u\in K_P.\end{equation*}
 \end{proposition}
\begin{proof}  Proposition \ref{p-summary}, Theorem \ref{t6} and Lemma \ref{l1} enable us to extend $\Psi$ to the required homeomorphism $\Phi\colon~K_P\to K_Q$.\end{proof}

It is a well known fact that the billiard map $T$ has a
natural invariant measure on its phase space $V_P$, the phase length given by the formula
$\mu=\sin\theta~\di x~\di\theta$ - see \cite{MT}.  In the case, when $P$ is rational and the corresponding
billiard flow is dense in the surface $R_P$,
 the measure $\mu$ sits on the
skeleton $K_P$ of $R_P$. In particular, an edge $e$
of $K_P$ associated with $\theta$ has the $\mu$-length $\vert
e\vert\cdot\sin\theta$.

For any
rational polygon with $N=2$ we can speak - up to rotation - about
horizontal, resp. vertical sides.  Two such polygons, $P$ and $Q$ with
sides $e_i$ resp.\ $f_i$, are
{\em affinely similar}  if they have the same number of corners/sides,
corresponding angles equal and there are positive numbers $a,b\in\bbr$ such that
$\vert e_i\vert/\vert f_i\vert=a$,
resp.\ $\vert e_i\vert/\vert f_i\vert=b$ for any pair of corresponding horizontal, resp. vertical sides.
Recall the map $\Phi$ defined in Proposition \ref{p3}.

As before the number $N$ is defined as the
least common multiple of $n_i$'s, where the the angles of a simply
connected rational polygon $P$ are $\pi m_i/n_i$.

\begin{theorem}\label{t5}Let $P$, $Q$ be code equivalent with leaders $u,v$; $P$ rational, $u$ non-exceptional.  Denote $\mu$, $\nu$ the phase length measure sitting on the skeleton $K_P$, $K_Q$ respectively. If $\nu=\Phi^*\mu$ then
\begin{enumerate}
\item{}  if $N=N_P\ge 3$, $Q$ is similar to $P$;
\item{}  if $N=N_P = 2$, $Q$ is affinely similar to $P$.
\end{enumerate}
\end{theorem}
\begin{proof}
We know from Theorem \ref{t6} that under our assumptions also $Q$ is rational with $v$
non-exceptional. By Lemma \ref{l4}, $N_P=N_Q$.

1) For a side $e$ of $P$ and a $\theta\in [-\frac{\pi}{2},\frac{\pi}{2}]$ denote $[e,\theta]$ an edge of $K_P$ associated with $e$ and $\theta$. Let $[f,\vartheta]=\Phi([e,\theta]))$. Since $\nu=\Phi^*\mu$ and $\mu,\nu$ are the phase lengths,
\begin{equation}\label{e-cosinus}\vert e\vert\sin\theta=\vert f\vert\sin\vartheta.\end{equation}
Assume that the least common multiple $N$ of the denominators of
angles of $P$ is greater than or equal to $3$. The polygons $P$, $Q$ correspond to the same
dihedral group $D_N$ generated by the reflections in lines through
the origin that meet at angles $\pi/N$. The orbit of
$\theta^+_0=\pi_2(u_0)$, resp.\  $\vartheta^+_0=\pi_2(u_0)$ under
$D_N$ consists of $2N$ angles
$$\theta_j^+=\theta_0^++2j\pi/N,~\theta_j^-=\theta_0^-+2j\pi/N,$$
resp.\
$$\vartheta_j^+=\vartheta_0^++2j\pi/N,~\vartheta_j^-=\vartheta_0^-+2j\pi/N.$$
Since $N\ge 3$, for each side $e$, resp.\  $f$ one can consider the
angles
$$\theta,\theta+2\pi/N,\text{ resp. }\vartheta,\vartheta+2\pi/N$$ such
that by Lemma \ref{l1} $\Phi[e,\theta]=(f,\vartheta)$
and $\Phi[e,\theta+2\pi/N]=[f,\vartheta+2\pi/N]$. Then as in
(\ref{e-cosinus}),
\begin{equation*}\label{e-cosina}\vert e\vert\sin\theta=\vert f\vert\sin\vartheta,~\vert e\vert\sin(\theta+2\pi/N)=
\vert f\vert\sin(\vartheta+2\pi/N),\end{equation*} hence after some
routine computation we get $\vert e\vert=\vert f\vert$.

2)  By Theorem \ref{p-sameangles} the polygons $P$ and $Q$ are quasisimilar hence we can speak about corresponding horizontal, resp. vertical sides. Similarly as above, for a side $e$ of $P$, some $\theta\in [-\frac{\pi}{2},\frac{\pi}{2}]$
and $[f,\vartheta]=\Phi([e,\theta])$,
\begin{equation*}\label{e-cosin}\vert e\vert\sin\theta=\vert f\vert\sin\vartheta,\end{equation*}
where $\theta$, resp. $\vartheta$ can be taken the same for any pair of corresponding horizontal, resp. vertical sides. Thus, the number $a=\vert e\vert/\vert f\vert$, resp.  $b=\vert e\vert/\vert f\vert$ does not depend on a concrete choice of a pair of corresponding horizonal, resp. vertical sides. This finishes the proof of our theorem.
\end{proof}

In a rational polygon we say that a point $u$ is
generic if it is non-exceptional, has bi-infinite orbit
and the billiard  map  restricted to the skeleton $K_P$ of an invariant surface $R_P\sim R_{\pi_2(u)}$
has a single invariant measure (this measure is then automatically the
measure $\mu$).

\begin{corollary}\label{cc1}Let $P$, $Q$ be code equivalent with leaders $u,v$; $P$ rational, $u$ generic.  Then
\begin{enumerate}
\item{}  if $N=N_P\ge 3$, $Q$ is similar to $P$;
\item{}  if $N=N_P = 2$, $Q$ is affinely similar to $P$.\end{enumerate}\end{corollary}
\begin{proof}Obviously the dynamical systems $(K_P,T)$, $(K_Q,S)$ are conjugated via the conjugacy $\Phi$, hence by our assumption on the element $u$, both of them are uniquely ergodic. It means that  $\nu=\Phi^*\mu$, where $\mu ,\nu$ are the phase lengths and Theorem \ref{t5} applies.\end{proof}

\section{Code versus order equivalence}\label{scoe}

In \cite{BT} we have defined another kind of equivalence relation on the set of simply connected polygons.
Namely, we used

\begin{definition}\label{d2}
We say that polygons (or polygonal billiards) $P,Q$ are order equivalent if for some $u\in FV_P$, $v\in FV_Q$
\begin{itemize}
   \item[(O1)] $\overline{\{\pi_1(T^nu)\}}_{n\ge 0}=\partial  P$, $\overline{\{\pi_1(S^nv)\}}_{n\ge
   0}=\partial  Q$,
     \item[(O2)] the sequences $\{\pi_1(T^nu)\}_{n\ge 0}$, $\{\pi_1(S^nv)\}_{n\ge 0}$
  have the same combinatorial order;
  \end{itemize}
  the points $u,v$ will be called leaders.
 \end{definition}

It is easy to see that any two rectangles are order equivalent.

Let $t=\{x_n\}_{n\ge 0}$ be a sequence which is dense in $\partial  P$. The
$t$-address $a_t(x)$ of a point $x\in\partial  P$ is the set of  all increasing
sequences $\{n(k)\}_k$ of non-negative integers satisfying
$\lim_kx_{n(k)}=x$. It is clear that any $x\in\partial  P$ has a nonempty
$t$-address and $t$-addresses of two distinct points from $\partial  P$ are
disjoint.

For order equivalent polygons $P$, $Q$  with leaders $u,v$, we
will consider addresses with respect to the sequences given by
Definition \ref{d2}(O2):
\begin{equation*}t=\{\pi_1(T^nu)\}_{n\ge 0},~s=\{\pi_1(S^nv)\}_{n\ge 0}.\end{equation*}

 It is an easy exercise to prove that the
map $\phi\colon\partial  P\to \partial  Q$ defined by
\begin{equation}\label{e18}\phi(x)=y~\text{ if }a_t(x)=a_s(y)\end{equation}
is a homeomorphism.

 As before, the set of the corners $p_1,\dots,p_k$ of $P$ is denoted by $C_P$.

\begin{theorem}\label{t3} Suppose $P$, $Q$ are order equivalent with leaders $u,v$; $P$ rational, $u$ non-exceptional.
Then $P$, $Q$ are code equivalent with leaders $u,v$. \end{theorem}
\begin{proof}It was shown in \cite[Theorem 4.2, Lemma 3.3]{BT} that $Q$ is rational, $v$ is non-exceptional and $\phi(C_P)=C_Q$, hence $\phi$ preserves also the sides:\begin{equation*}\label{e2}\phi([p_i,p_{i+1}])=[q_i,q_{i+1}],~i=1,\dots,k.\end{equation*} Since by (\ref{e18}) for the leaders $u,v$
\begin{equation*}\label{e13-globalphi}\phi(\pi_1(T^nu))=\pi_1(S^nv),\end{equation*}
the symbolic forward itineraries $\sigma^+(u)$, $\sigma^+(v)$ are the same.
\end{proof}

\begin{theorem}\label{converse}
Suppose $P,Q$ are code equivalent with leaders $u,v$; $P$ rational, $u$ generic.  Then $P,Q$ are order equivalent with leaders $u,v$.
\end{theorem}

\begin{proof}
Apply Corollary \ref{cc1}, then $P$ and $Q$ are similar (or affinely
 similar). Let $\Phi$ be the map defined in Proposition \ref{p3}.
 
\vskip.1mm\noindent$N=3$. By Proposition \ref{p3}, $\Phi(u) = v$. Since $P$ and
$Q$ are similar, Theorem \ref{t2} implies that $v=u$ (up to similarity) for the same code of $u,v$ hence $P$, $Q$ are order equivalent with leaders $u,v$.
\vskip.1mm\noindent$N=2$. Arguing as in the proof of Corollary \ref{cc1} we get $\nu=\Phi^*\mu$, where $\mu ,\nu$ are the phase 
lengths. Now, on different edges $k_1=[a_1,b_1]$, $k_2=[a_2,b_2]$ of $K_P$ that correspond to the same side $[a,b]$ of $P$ the proportions given by $\mu$ are
 preserved, i.e., for $\mu_i=\mu\vert k_i$ and each $x\in (a,b)$ and corresponding $x_i\in k_i$,
 \begin{equation*}\mu_i([a_i,x_i])/\mu_i(k_i)=\lambda([a,x])/\lambda([a,b]).\end{equation*} 
Since $\nu=\Phi^*\mu$ and $\nu$ is the phase length, on $\ell_i=\Phi(k_i)$ the proportions given by $\nu$ are
also preserved. It means that the sequences $\{\pi_1(T^nu)\}_{n\ge 0}$, $\{\pi_1(S^nv)\}_{n\ge 0}$
  have the same combinatorial order and $P$, $Q$ are order equivalent with leaders $u,v$. 
\end{proof}

\begin{corollary}\label{equiv}
Suppose $P$ is a rational polygon and $u \in FV_P$ is generic. Then $P,Q$  are code equivalent with leaders $u,v$
if and only if  $P,Q$ are order equivalent with leaders $u,v$.
\end{corollary}
\begin{proof}
It follows from Theorems \ref{t3} and \ref{converse}.
\end{proof}

{\bf Acknowledgements}  We gratefully acknowledge the support of the
``poste tcheque'' of the University of Toulon. The first author was
also supported by MYES of the Czech Republic via contract MSM
6840770010. The second author was supported by ANR-10-BLAN 0106 Perturbations.

\end{document}